\documentclass{article}

\usepackage{arxiv}
\usepackage{setspace}
\usepackage{amssymb,amsthm,amsmath,hyperref,txfonts}
\usepackage[utf8]{inputenc} % allow utf-8 input
\usepackage[T1]{fontenc}    % use 8-bit T1 fonts
\usepackage{hyperref}       % hyperlinks
\usepackage{url}            % simple URL typesetting
\usepackage{booktabs}       % professional-quality tables
\usepackage{amsfonts}       % blackboard math symbols
\usepackage{nicefrac}       % compact symbols for 1/2, etc.
\usepackage{microtype}      % microtypography
\usepackage{lipsum}
\usepackage{xcolor}
\usepackage{graphicx}
\usepackage{tikz-cd}
\usepackage{comment}
\usepackage{bm}
\usepackage[
    backend=biber,
    style=alphabetic,
  ]{biblatex}

\graphicspath{ {./images/} }

\newtheorem{theorem}{Theorem}[section]
\theoremstyle{definition}
\newtheorem{definition}[theorem]{Definition}
\newtheorem{remark}[theorem]{Remark}
\newtheorem{assumption}[theorem]{Assumption}
\newtheorem{lemma}[theorem]{Lemma}

\newcommand{\T}{\mathbb{T}}
\newcommand{\tb}[1]{\boldsymbol{#1}}

\DeclareMathOperator*{\argmin}{argmin}
\title{Learning operators for identifying weak solutions to the Navier-Stokes equations}
\author{
 Dixi Wang \\
  Department of Mathematics\\
  University of Florida\\
  \texttt{dixiwang@ufl.edu} \\
   \And
 Cheng Yu \\
  Department of Mathematics\\
  University of Florida\\
  \texttt{chengyu@ufl.edu} \\
}
\begin{document}
\setstretch{1.2}
\maketitle

\begin{abstract}
    This paper focuses on investigating the learning operators for identifying weak solutions to the Navier-Stokes equations. Our objective is to establish a connection between the initial data as input and the weak solution as output. 
    To achieve this, we employ a combination of deep learning methods and compactness argument to derive learning operators for weak solutions for any large initial data in 2D,  and for low-dimensional initial data in 3D. 
Additionally, we utilize the universal approximation theorem to derive a lower bound on the number of sensors required to achieve accurate identification of weak solutions to the Navier-Stokes equations.
Our results demonstrate the potential of using deep learning techniques to address challenges in the study of fluid mechanics, particularly in identifying weak solutions to the Navier-Stokes equations.
\end{abstract}

\section{Introduction}

The incompressible Navier-Stokes equations on the space-time domain $\mathbb{T}^d\times [0,T]$ is written as:
\begin{align} \label{ICNS}
    &\partial_t \tb{u} + \tb{u}\cdot \nabla \tb{u} + \nabla p -\nu\Delta \tb{u} = 0 ~~~~~~~~~~~~ (\tb{x},t)\in\mathbb{T}^d\times[0,T],\\
    &\text{div}~\tb{u}=0 ~~~~~~~~~~~~~~~~~~~~~~~~~~~~~~~~~~~~~~~~~~~(\tb{x},t)\in\mathbb{T}^d\times[0,T],
\end{align}
with initial data
\begin{align} \label{ini_ICNS}
    \tb{u}(\cdot,t)|_{t=0} = \tb{u}_0~~~~~\text{in}~\mathbb{T}^d.
\end{align}
Here $\mathbb{T}^d=\mathbb{R}^d/2\pi\mathbb{Z}^d$ is a periodic domain with $d=2~\text{or}~3$, $T>0$. 
$\tb{u}(\tb{x},t)$ is the velocity field, $p(\tb{x},t)$ is the pressure, and 
$\nu>0$ is the viscosity  coefficient.

The Navier-Stokes equations are fundamental in the fields of science and engineering. Despite extensive research, the mathematical understanding of the incompressible Navier-Stokes (ICNS) remains incomplete. One of the most significant challenges is the lack of global-in-time smooth solutions for ICNS in three dimensions, which has been a long-standing open problem. While global Leray-Hopf weak solutions have been proven to exist for dimensions $d\geq 2$, uniqueness has only been established in 2D, leaving the uniqueness of 3D Leray-Hopf solutions an open problem \cite{PLLions}. These unresolved issues highlight the complexity of the Navier-Stokes equations and the need for continued research in this area. Therefore, this paper focuses on investigating the learning operators for weak solutions to the Navier-Stokes equations, aiming to contribute to the understanding and development of this critical field.

Deep neural networks, as \textit{universal approximators} can approximate finite-dimensional continuous functions up to arbitrary accuracy \cite{barron_universal_1993}.  
In recent years, we have witnessed tremendous success of network-based PDE solvers working with smooth solutions in explicit form with high accuracy. For example, \textit{physics
informed neural networks} (PINNs) can solve both forward problems as well as inverse problems \cite{Raissi_2018},\cite{Raissi_2019}. When the solution attains singularities, \cite{WAN} parameterized both the weak solution and the test function by neural networks and learn the parameters in an adversarial manner governed by weak formulation. In \cite{Friedrich}, the authors proposed the Friedrichs learning to transform the PDE problem into a minmax optimization problem, their algorithm is capable to capture the discontinuity of the solution without a priori knowledge.

Although deep learning has achieved huge success in many applications, its statistical theory, particularly on operator approximations between infinite-dimensional spaces, is still limited. The universal approximation theorem, which allows networks to approximate operators between two compact subsets of infinite-dimensional Banach spaces, was first proved by \cite{chen_universal_1995}. Later  \cite{lu_deeponet_2021} introduced DeepONets, a combination of branch and trunk nets, that can approximate a continuous operator mapping from a compact subset of an infinite-dimensional space to another Banach space within any desired precision. However, neither of these results provides explicit estimation on the complexity of the networks involved.
In 2017, Yarotsky  \cite{yarotsky_error_2017} described a ReLU network architecture with a certain depth and complexity that can approximate any functions within the unit ball of the Sobolev space  $\mathcal{W}^{n,\infty}([0,1]^d)$  to any desired precision. \cite{liu_deep_2022} derived an upper bound on the generalization error for learning Lipschitz operators between infinite-dimensional Hilbert spaces by neural networks, which reveals a direct relationship between generalization error and the number of data samples. \cite{lanthaler_error_2022} weakened the assumptions on continuity and compactness of the input space and provided explicit bounds for DeepONets. \cite{jiao_deep_2023} derived an error bound for Hölder continuous functions, with a prefactor in the error bound dependent on the dimension in polynomial order.  Overall, these findings provide insights into the capabilities and limitations of neural networks in approximating operators between infinite-dimensional spaces and shed light on the complexity of the networks required for different levels of precision.

Both \cite{liu_deep_2022} and \cite{lanthaler_error_2022} employ a model reduction method, which is a common technique in data-driven learning for time-dependent parametrized PDEs as documented in the literature \cite{chinesta, bhattacharya}. In order to approximate maps between infinite-dimensional spaces, a typical first step is to reduce the dimensions of both input and output spaces using encoders and decoders (also known as projectors and reconstructors). Encoders are used to project high-dimensional data onto a lower-dimensional space, while decoders are used to map the reduced data back to the original space. Popular choices of encoders include orthogonal basis of a separable Hilbert space and PCA (principal component analysis) encoder. Basis encoders, such as orthogonal polynomials and trigonometric basis, are independent of the training data and only depend on the underlying Hilbert space. In contrast, PCA encoder is data-driven and has become increasingly popular in recent years due to its simplicity of implementation and good numerical performance \cite{shawe-taylor_eigenspectrum_2005, blanchard_statistical_2007}.

The Navier-Stokes equations present challenges for accurately capturing weak solutions using traditional numerical methods in various settings. Deep neural networks, on the other hand, offer a highly flexible approach for approximating solutions to some PDEs under different settings.  Hence, in this paper, we explore the use of deep learning operators to identify weak solutions of the Navier-Stokes equations, which can provide a more efficient and accurate approach to fluid dynamics problems.

%Inspired and built on \cite{bhattacharya} and \cite{lu_deeponet_2021}, we aim to learn a nonlinear operator mapping from the observed initial function space to the space of Leray-Hopf solution to (\ref{ICNS})-(\ref{ini_ICNS}) in both 2D and 3D cases via nonparametric regression estimation based on deep neural networks.  
%We also estimate the number of points (\textit{sensors}) needed in $[-\pi,\pi]^3$ to obtain a good accuracy in 3D case. 
%In \cite{lu_deeponet_2021}, The authors treated initial functions as input functions of an unknown operator and estimated the number of sensors when the input functions being Gaussian process. Later \cite{lanthaler_error_2022} analyzed the approximation error of DeepONet, and projected input functions from infinite dimensional space into an $m$-dimensional vector space, where $m$ denotes the number of sensors. In particular, $m\gg 1$. The high dimensionality of the input vector space affects the efficiency in realistic learning tasks. We adopt the concept of \textit{sensor} in a different setting: the input functions are encoded using eigenfunctions of Stokes operator into a $D$-dimensional vector space, $D$ is independent of $m$. Given an accuracy, we estimate the number of sensors based on the \textit{universal approximation theorem} \cite{chen_universal_1995}.

The current paper builds upon and is inspired by \cite{bhattacharya} and \cite{lu_deeponet_2021}, the aim of this paper is to use deep neural networks to learn a nonlinear operator that maps from an observed initial function space to the space of Leray-Hopf solutions to (\ref{ICNS})-(\ref{ini_ICNS}) in both 2D and 3D cases via nonparametric regression estimation. We also aim to estimate the number of sensors required to achieve good accuracy in the 3D case on $[-\pi,\pi]^3$.
In \cite{lu_deeponet_2021}, the authors treated initial functions as input functions of an unknown operator and estimated the number of sensors required when the input functions were modeled as Gaussian processes. Later, in \cite{lanthaler_error_2022}, the approximation error of DeepONet was analyzed, and input functions were projected from an infinite-dimensional space into an $m$-dimensional vector space, where $m$ is the number of sensors. However, the high dimensionality of the input vector space can limit the efficiency of learning in realistic tasks.
To address this issue, we adopt a different approach by encoding the input functions using eigenfunctions of the Stokes operator into a $D$-dimensional vector space, where $D$ is independent of $m$. Given a desired level of accuracy, we estimate the number of sensors required based on the \textit{universal approximation theorem} \cite{chen_universal_1995}.

The rest of the paper is organized as follows: In Section \ref{sec:notation}, we define weak solutions and deep neural networks. We introduce the learning framework by specifying the encoder and decoder for our work and making several assumptions on the operator learning problem,  Then we state our main results. In Section \ref{sec:main}, we first briefly describe the Galerkin approximation then we prove a key Lipschitz property of the solution operator at each Galerkin level. Based on this Lipschitz continuity, we prove a lemma on the network approximation error. The proof of Theorem \ref{thm:ICNS_2D} and Theorem \ref{thm:ICNS_3D_lowdim} follows. Finally in Section \ref{sec:data}, we prove Theorem \ref{thm:low bound} on an estimation of number of sensors needed to obtain a good accuracy.

\section{Notation, learning setting and statement of main results}\label{sec:notation}
\subsection{Weak solutions}
We briefly introduce Leray-Hopf weak solutions to the Navier-Stokes equations.
For convenience, we define the following function spaces:
\begin{align*}
&\mathcal{V} = \{\tb{\varphi}\in C_0^\infty (\T^d)~|~\text{div}~\tb{\varphi}=0\},\\
&H = \text{closure of}~ \mathcal{V} ~ \text{in}~ L^2(\T^d) ,\\
&V_s = \text{closure of}~\mathcal{V}~ \text{in}~H^s(\T^d),~~s>0.
\end{align*}

\begin{remark}
The space $V_s$ is equipped with the norm $\|\cdot\|_{V_s}=\|\cdot\|_{H^s_0(\T^d)}$. In particular, if $s=1$, denote $V=V_1$, and one can identify the space 
$V = \{v\in {H^1_0(\T^d)}~|\text{ div}~v = 0\}.
$
\end{remark}
The weak solution is given in the following sense:
\begin{definition}\label{def:leray}
For any $T>0$, $\tb{u}(t,\tb{x})$ is a Leray-Hopf weak solution of system \ref{ICNS}-\ref{ini_ICNS} if
\[
\tb{u}\in L^2(0,T;V) \cap L^\infty(0,T;H),~~\text{and}~\int_{\T^d} \tb{u}~\text{d}\tb{x} = 0,
\]
satisfies
\begin{align}\label{Leray}
\begin{split}&
    -\int^T_0 \int_{\T^d} \tb{u}\cdot \partial_t\tb{\phi}~ \text{d}\tb{x}\text{d}t + \int^T_0 \int_{\T^d} (\tb{u}\cdot \nabla)\tb{u}\cdot\tb{\phi}~ \text{d}\tb{x}\text{d}t +\nu\int^T_0 \int_{\T^d} \nabla \tb{u}:\nabla \tb{\phi} ~ \text{d}\tb{x}\text{d}t 
    \\&\quad\quad\quad\quad = \int_{\T^d} \tb{u}_0\cdot \tb{\phi}(0)~\text{d}\tb{x},~~~\text{for all}~~\tb{\phi}\in C^\infty_c(\mathbb{T}^d\times [0,T))\;\;\text{with}~ \text{div}\tb{\phi}=0 ,
    \end{split}
\end{align}
and the energy inequality
\begin{align}\label{energy_equality}
    \dfrac{1}{2}\int_{\T^d} |\tb{u}(\tb{x},t)|^2~\text{d}\tb{x} + \nu\int^T_0 \int_{\T^d} |\nabla \tb{u}(\tb{x},t)|^2 \text{d}\tb{x}\text{d}t \leq \dfrac{1}{2}\int_{\T^d} |\tb{u}_0(\tb{x})|^2~\text{d}\tb{x}.
\end{align}
\end{definition}

\subsection{Learning framework}

A feedforward neural network $f: \mathbb{R}^n \rightarrow \mathbb{R}$ is written as:
\begin{equation}\label{def:fnn}
f(\tb{x}) = W_L \cdot \sigma( W_{L-1}\cdots\sigma(W_1 \tb{x} + b_1)+\cdots+b_{L-1}) + b_L, ~~ \tb{x}\in\mathbb{R}^n,
\end{equation}
where $W_i$ and $b_i$ are the weight matrix and bias vector for layer $i=1,\dots,L$. $L$ denotes the depth of the network, and $\sigma$ is an activation function applied element-wise to the intermediate outputs. In this paper, we focus on neural networks using the rectified linear unit (ReLU) activation function:
\begin{equation*}
\sigma(\tb{x}) = (\max\{0,x_1\}, \max\{0,x_2\},\dots,\max\{0,x_n\}), ~~ \text{for}~\tb{x}=(x_1,\dots,x_n)\in\mathbb{R}^n.
\end{equation*}

We consider a network class $\mathcal{F}_{NN}$ that contains ReLU neural networks mapping from a compact domain of the input space $[-M,M]^n\subset\mathbb{R}^n$ to the output space $\mathbb{R}^{n'}, M>0$ as follows:
%In this setting, the constraints on cardinality and magnitude of weight parameters are not required as long as the sup-norm of each $f_k$ is bounded by some $M>0$. 
\begin{align}\label{def:FNN}
    \mathcal{F}_{NN}(n,n',L,r,M) &= \{\Gamma =[f^{(1)},f^{(2)},\dots,f^{(n')}]^T: \text{for each}~k=1,\dots,n',\nonumber\\
    & f^{(k)}(\tb{x}):[-M,M]^n\to\mathbb{R}~ \text{is in the form}~ (\ref{def:fnn})~\text{with}~L^{(k)}\leq L,~r^{(k)}\leq r
    \},
\end{align}
where $r^{(k)}$ denotes the number of weights and bias of $f^{(k)}$.

\subsection{Our learning setup}

Let $\mathbb{P}$ be a Borel probability measure supported on $H$, 
%, we assume that $\mathbb{P}$ is supported on a compact set $K\subset H$.
Let $S(N)=\{\tb{x}_i,\tb{y}_i\}^N_{i=1}$ be a training set from which $\{\tb{x}_i\}^N_{i=1}$ is a finite collection of $N$ i.i.d. draws from $\mathbb{P}$, and $\{\tb{y}_i\}^N_{i=1}$ are the corresponding output samples w.r.t. the push-forward measure $\Psi_\sharp \mathbb{P}$ and evaluated at a fixed time.
%Each $\tb{v}_i$ is added by a noise $\epsilon_i$. We assume that $\epsilon_i$ are i.i.d. sampled in the solution space $Y$, independent of $\tb{u}_i$. Moreover, $\mathbb{E}[\epsilon_i] = 0$ and there exists $\sigma>0$ such that $\|\epsilon_i\|_Y\leq \sigma$.
The training data is generated using pseudo-spectral methods.
%We aim to seek a learning operator that can approximate the weak solutions to the Naiver-Stokes equations using a finite training set $S(N)$ at a fixed time.
We denote by $Y:= L^2(0,T;V)$ as the space of outputs to simplify the notation and let $\Psi:H\to Y$ be the unknown operator mapping initial functions to the Leray-Hopf solutions of the ICNS. 
Instead of a direct attempt to approximating $\Psi$ between infinite-dimensional spaces, we exploit a finite-dimensional structure by reducing the dimensions of the input and output spaces using encoders and decoders.

The Galerkin approximation in Section \ref{sec:Galerkin} provides a natural choice of encoders and decoders for the underlying PDE. 
Let $d_H>0$ be the dimension of the encoded vector space for the input space $H$, we define the following induced encoder $\mathcal{E}^{d_H}:H\to\mathbb{R}^{d_H}$:
\begin{align}\label{encoder}
    \mathcal{E}^{d_H}(\tb{x}):=(\langle \tb{x}, \tb{w}_1\rangle, \langle \tb{x}, \tb{w}_2\rangle, \dots,\langle \tb{x},\tb{w}_{d_H}\rangle )^T,
\end{align}
where $\{\tb{w}_k\}^{d_H}_{k=1}$ are the first $d_H$ eigenfunctions of the Stokes operator, and $\langle \cdot,\cdot \rangle$ denotes the inner product on the underlying space. 
%And $\mathcal{E}^D(u_0)$ is a vector containing the first $D$ coefficients of the decomposition of $u_0$ using $\{w_k\}^D_{k=1}$. 
Conversely, we define the decoder $\mathcal{D}^{d_H}:\mathbb{R}^{d_H}\to H$ by 
\begin{align}\label{decoder}
    \mathcal{D}^{d_H}(\textbf{a}) := \sum^{d_H}_{i=1} a_i \tb{w}_i, ~~\text{for any}~~\textbf{a} = (a_1,a_2,\dots,a_{d_H})\in\mathbb{R}^{d_H}.
\end{align}
Denote $\Pi_{d_H}:= \mathcal{D}^{d_H}\circ\mathcal{E}^{d_H}$. It should be satisfied that
\begin{align}\label{Pi_H}
    \mathcal{E}^{d_H}\circ\mathcal{D}^{d_H} = \text{Id}_{\mathbb{R}^{d_H}}: \mathbb{R}^{d_H}\to \mathbb{R}^{d_H},~~~\text{and}~~~
    \Pi_{d_H} \approx \text{Id}_H: H\to H.
\end{align}

Suppose the dimension of the encoded vector space for the solution space $Y$ is $d_Y$. We define the encoder $\mathcal{P}^{d_Y}:Y\to\mathbb{R}^{d_Y}$ and the decoder $\mathcal{R}^{d_Y}:\mathbb{R}^{d_Y}\to Y$ similarly:
\begin{align}\label{Y:encoder}
    \mathcal{P}^{d_Y}(\tb{y}) &:=(\langle \tb{y}, \tb{w}_1\rangle, \langle \tb{y}, \tb{w}_2\rangle, \dots,\langle \tb{y},\tb{w}_{d_Y}\rangle )^T \in \mathbb{R}^{d_Y}, ~~\text{for any}~~\tb{y}\in Y, \nonumber\\
    \mathcal{R}^{d_Y}(\textbf{b}) &:= \sum^{d_Y}_{i=1} b_i \tb{w}_i, ~~\text{for any}~~\textbf{b} = (b_1,b_2,\dots,b_{d_Y})\in\mathbb{R}^{d_Y}.
\end{align}
Here, we regard each $b_i$ as an evaluation of $d^{d_Y}_i(t)$ defined in (\ref{ICNSdef:u_m}) at a fixed time.
Denote $\Pi_{d_Y}:=\mathcal{R}^{d_Y}\circ\mathcal{P}^{d_Y}$. Again, it should be satisfied that
\begin{align}\label{Pi_Y}
    \mathcal{P}^{d_Y}\circ\mathcal{R}^{d_Y} = \text{Id}_{\mathbb{R}^{d_Y}}: \mathbb{R}^{d_Y}\to \mathbb{R}^{d_Y},~~~\text{and}~~~
    \Pi_{d_Y} \approx \text{Id}_Y: Y\to Y.
\end{align}

We have the following lemma for the properties of these operators, which is necessary in our proof later.
\begin{lemma}\label{encoder lips}
    $\mathcal{E}^{d_H}$, $\mathcal{D}^{d_H}$, $\mathcal{P}^{d_Y}$, $\mathcal{R}^{d_Y}$ are Lipschitz operators with Lipschitz coefficients 1.
\end{lemma}

\begin{proof}
    Let $\tb{x}_1,\tb{x}_2\in H$. Since $H$ is a separable Hilbert space and the eigenfunctions $\{\tb{w}_k\}^\infty_{k=1}$ of the Stokes operator $-\Delta$ form an orthonormal basis of $H$, we have
\begin{align*}
    \|\mathcal{E}^{d_H}(\tb{x}_1)-\mathcal{E}^{d_H}(\tb{x}_2)\|^2_2 = \sum^{d_H}_{i=1}|\langle \tb{x}_1-\tb{x}_2,\tb{w}_i\rangle |^2 \leq \sum^\infty_{i=1}|\langle \tb{x}_1-\tb{x}_2,\tb{w}_i\rangle |^2 = \|\tb{x}_1-\tb{x}_2\|^2_H.
\end{align*}
For any $\tb{a},\Bar{\tb{a}}\in \mathbb{R}^{d_H}$,
\begin{align*}
    \|\mathcal{D}^{d_H}(\tb{a})-\mathcal{D}^{d_H}(\Bar{\tb{a}})\|^2_H = \|\sum^{d_H}_{i=1}(a_i-\Bar{a}_i)w_i\|^2_H = \|\tb{a}-\Bar{\tb{a}}\|^2_2.
\end{align*}
A similar argument gives the Lipschitz continuity of $\mathcal{P}^{d_Y}$, $\mathcal{R}^{d_Y}$ and their Lipschitz constants 1.
\end{proof}

Once we define the encoders and decoders, a finite-dimensional map on the discretization grid can be learned by deep neural networks. The whole process is illustrated in the following diagram:

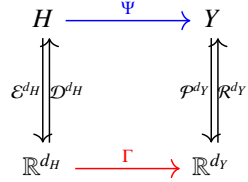
\begin{center}
\begin{tikzcd}[row sep=4em,column sep=4em]
    H \arrow[r,blue, "\Psi",blue] \arrow[d, "\mathcal{D}^{d_H}"]
& Y \arrow[d, "\mathcal{R}^{d_Y}"] \\
\mathbb{R}^{d_H} \arrow[r, red, "\Gamma" red]
\arrow[u,shift right,"\mathcal{E}^{d_H}"]
& \mathbb{R}^{d_Y}\arrow[u,shift right,"\mathcal{P}^{d_Y}"]
 \end{tikzcd}
\end{center}

%Since the exact solution is not available in our work, given $\tb{u}_i\in H$ i.i.d. serving as initial functions, we use Chebfun package \cite{Chebfun} together with explicit Runge-Kutta temporal integrator of suitable time-step size to generate $\tb{v}_i$.
%Then we add an artificial noise $\epsilon_i$ to $\tb{v}_i$. 

%A minimizer $\Gamma_{NN}\in \mathcal{F}_{NN}$ is chosen to minimize the following empirical mean squared error (or $L_2$ risk) 
%\begin{align}\label{minimizer}
%    \Gamma_{NN} \in \argmin_{\Gamma\in\mathcal{F}_{NN}}\dfrac{1}{n}\sum^n_{i=1}\|\Gamma\circ E^D(u_i) - E^K(v_i)\|^2_2.
%\end{align}

\subsection{Main results}

We state our main results in this subsection. First, we need following assumptions.
\begin{assumption}\label{ass:energy}
For every initial function $\tb{u}_0\in H$, 
\begin{align}\label{icns:initial_energy}
        \|\tb{u}_0\|_H\leq R,~~ \text{for some}~R\in\mathbb{R}^+.
    \end{align}
\end{assumption}
Thus initial functions are chosen from the ball $B_R(H):=\{\tb{x}\in H;~\|\tb{x}\|_H\leq R\}$.
In 3D case, we need the following restriction on the initial function:
\begin{assumption}\label{ass:lowdim}
    Let $\{\tb{w}_k\}^\infty_{k=1}$  be the orthonormal basis of $H$. There exists $D\in\mathbb{N}^+$  such that every initial function $\tb{u}_0$ satisfies $$\tb{u}_0\in \text{span}\{\tb{w}_1,\tb{w}_2,\dots\,\tb{w}_D\}\subset H.$$ We call such initial data as low-dimensional initial data.
\end{assumption}

\begin{remark}
    As one shall see in Lemma \ref{ICNS:lips},
    Assumption \ref{ass:lowdim} imposes an uniform upper bound on the Lipschitz constant $L_{\Psi^k}, k=1,2,\dots$, in 3D case so that both the network approximation error and the projection error can be controlled given an accuracy $\epsilon$. The low-dimensional structure has been proposed in several examples to improve the convergence rate significantly (\cite{baraniuk_random_2009},\cite{jiao_deep_2023}), and has been observed in applications such as image processing. In \cite{chinesta}, the authors analyzed the reduced basis method on various PDE models, and discussed when it leads to a good approximation error.   
\end{remark}

For 2D Leray-Hopf solutions, we have the following result:

%\begin{theorem}\label{thm:ICNS_2D}

%Suppose all the assumptions on the probability measure $\mathbb{P}$ and the training data $S(N)$ hold, and the encoders and decoders are given in (\ref{encoder}),(\ref{decoder}) and (\ref{Y:encoder}). 
%Let $u_0\in H$ denote the initial condition uniformly bounded as in (\ref{icns:initial_energy}). 
%Let $\Psi:H\to Y$ be a $\mathbb{P}$-measurable operator mapping 
%$u_0\in H$ to the Leray-Hopf weak solution to the 2D incompressible Navier Stokes equation (\ref{ICNS})-(\ref{ini_ICNS}), where $u_0$ is a initial function satisfying Assumption \ref{ass:energy}. 
%Let $\epsilon>0$, 
%there exist a network class $\mathcal{F}_{NN}(N,L,p,M)$ in the form (\ref{def:FNN}),and a network $\Gamma_{NN}:\mathbb{R}^D\to \mathbb{R}^N$ in $\mathcal{F}_{NN}$ such that 
%\begin{align}\label{ineq:thm_icns2D}
%    \mathbb{E}_{\mathcal{S}(N)}\mathbb{E}_{u_0\sim\mathbb{P}}\|D^N\circ \Gamma_{NN}\circ E^D(u_0) - \Psi(u_0)\|^2_Y < \epsilon
%\end{align}
%for a requisite amount of data $n$, and some fixed $D, N\in\mathbb{N}\backslash \{0\}$ sufficient large.

%\end{theorem}

\begin{theorem}\label{thm:ICNS_2D}
    %Suppose all the assumptions on the probability measure $\mathbb{P}$ hold. Let $\{\tb{x}_i\}^N_{i=1}$ be $N$ i.i.d. draws from $\mathbb{P}$ satisfying Assumption \ref{ass:energy} and
    Let $\mathbb{P}$ be a Borel probability measure supported on $B_R(H)$ and 
    $\Psi:H\to Y$ be a Borel measurable operator mapping initial function $\tb{x}\in B_R(H)$ to the Leray-Hopf weak solution to the 2D incompressible Navier-Stokes equations (\ref{ICNS})-(\ref{ini_ICNS}). Let $\epsilon>0$, there exist $d_H(L_\Psi,\epsilon), d_Y(\epsilon)>0, N=N(d_H,d_Y)$, a ReLU network $\Gamma_{NN}\in\mathcal{F}_{NN}(d_H,d_Y,L,r, \frac{R}{\sqrt{d_H}})$ with $L=L(d_H,d_Y,L_\Psi,R,\epsilon)$ and $r = r(d_H,d_Y,L_\Psi,R,\epsilon)$ such that 
    \begin{align}\label{ineq:thm_icns2D}
        \mathbb{E}_{\tb{x}_i\sim\mathbb{P}}\mathbb{E}_{\tb{x}\sim\mathbb{P}}\|\mathcal{R}^{d_Y}\circ \Gamma_{NN} \circ \mathcal{E}^{d_H}(\tb{x}) - \Psi(\tb{x})\|^2_Y < \epsilon.
    \end{align}
    %where $\mathcal{R}^{d_Y}$, $\mathcal{E}^{d_H}$ are defined in (\ref{Y:encoder}), (\ref{encoder}), and $L=L(d_H,d_Y,L_\Psi,R,\epsilon)$, $r = r(d_H,d_Y,L_\Psi,R,\epsilon)$.
\end{theorem}

%Under Assumption \ref{ass:lowdim}, it is clear that the dimension of encoding vector space is $D$. Follow the construction of $u^k$ for each Galerkin order $k$, $u^k(t)$ lies in $\text{span}\{w_1,w_2,\dots\,w_D\}$. Thus the dimension of the decoding vector space is also $D$. Under the low dimensional assumption,
For 3D Leray-Hopf solutions, we have the following result:

\begin{theorem}\label{thm:ICNS_3D_lowdim}
    Suppose Assumption \ref{ass:lowdim} holds. Let $\mathbb{P}$ be a Borel probability measure supported on $B_R(H)$  and $\Psi:H\to Y$ be a Borel measurable operator mapping initial function $\tb{x}\in B_R(H)$ to the Leray-Hopf weak solution to the 3D incompressible Navier-Stokes equations (\ref{ICNS})-(\ref{ini_ICNS}). Let $\epsilon>0$, there exist $d_H(L_\Psi,\epsilon)\geq D, d_Y(\epsilon)>0, N=N(d_H,d_Y)$, a ReLU network $\Gamma_{NN}\in\mathcal{F}_{NN}(d_H,d_Y,L,r, \frac{R}{\sqrt{d_H}})$ with $L=L(d_H,d_Y,L_\Psi,R,\epsilon)$ and $r = r(d_H,d_Y,L_\Psi,R,\epsilon)$ such that 
    \begin{align}\label{ineq:thm_icns3D}
        \mathbb{E}_{\tb{x}_i\sim\mathbb{P}}\mathbb{E}_{\tb{x}\sim\mathbb{P}}\|\mathcal{R}^{d_Y}\circ \Gamma_{NN} \circ \mathcal{E}^{d_H}(\tb{x}) - \Psi(\tb{x})\|^2_Y < \epsilon.
    \end{align}
    %where $\mathcal{R}^{d_Y}$, $\mathcal{E}^{d_H}$ are defined in (\ref{Y:encoder}), (\ref{encoder}) and $L=L(d_H,d_Y,L_\Psi,R,\epsilon)$, $r = r(d_H,d_Y,L_\Psi,R,\epsilon)$.
\end{theorem}
%\begin{theorem}\label{thm:ICNS_3D_lowdim}

%Suppose all the assumptions on the probability measure $\mathbb{P}$ and the training data $S(n)$ hold, and the encoder $E^D$ and the decoder $D^D$ are given in (\ref{encoder})-(\ref{decoder}). 

%Denote $\Psi:H\to Y$ a $\mathbb{P}$-measurable operator mapping 
%$u_0\in H$ to the Leray-Hopf solution to the incompressible Navier Stokes equation (\ref{ICNS})-(\ref{ini_ICNS}), where $u_0$ is an initial function satisfying Assumption \ref{ass:energy} and \ref{ass:lowdim}. 
%Let $\epsilon>0$, 
%there exist a network class $\mathcal{F}_{NN}(N,L,p,M)$ in the form (\ref{def:FNN}), a finite network $\Gamma_{NN}:\mathbb{R}^D\to \mathbb{R}^D$ in $\mathcal{F}_{NN}$ such that 
%\begin{align}\label{ineq:thm_icns}
%    \mathbb{E}_{\mathcal{S}(n)}\mathbb{E}_{u_0\sim\mathbb{P}}\|D^D\circ \Gamma_{NN}\circ E^D(u_0) - \Psi(u_0)\|^2_Y < \epsilon
%\end{align}
%for a requisite amount of data $n$.

%\end{theorem}
The double expectation averages over new inputs $\tb{x}$ drawn from $\mathbb{P}$ and over realizations of the i.i.d. data $S(N)$. Finally,
we give an estimation on data generation of the approximation theory to 3D Leray-Hopf solutions. 
One can choose uniformly $m+1$ points $\tb{x}_0,\tb{x}_1,\dots,\tb{x}_m \in [-\pi,\pi]^3$ as location of sensors. From Section \ref{sec:Galerkin}, we know $\tb{w}_i\in C^\infty_c(\T^3)$ for $i=1,2,\dots, D$. We define a function $\tb{w}^m_i(\tb{x})$ to be the trilinear interpolation of $\tb{w}_i(\tb{x})$ based on $\tb{x}_0,\tb{x}_1,\dots,\tb{x}_m$.
Using Taylor expansion, there exists a constant $\kappa(D,m)>0$ such that
\begin{align}
    \max_{\tb{x}\in [-\pi,\pi]^3} |\tb{w}_i(\tb{x})-\tb{w}^m_i(\tb{x})|\leq \kappa(D,m), 
\end{align}
where $\kappa(D,m)\to 0$ as $m\to\infty$.

\begin{theorem}\label{thm:low bound} 
Let $\epsilon>0$.
Suppose that Assumption \ref{ass:energy} and Assumption \ref{ass:lowdim} hold for some $R, D >0$ and $m$ is a positive integer making $$C_1e^{CR^4\lambda_D}\kappa(D,m)< \frac{1}{2}\epsilon,$$
where $C_1=C_1(D,R)>0$. Let $\Psi:H\to Y$ be the operator mapping initial function $\tb{u}_0\in B_R(H)$ to the Leray-Hopf weak solution to the 3D incompressible Navier-Stokes equations (\ref{ICNS})-(\ref{ini_ICNS}).
Then there exist $d_H, d_Y, k>0$ and a depth-2 ReLU neural network with weights and bias $W_1\in \mathbb{R}^{k\times d_H}$, $b_1\in\mathbb{R}^{d_H}$, $W_2\in\mathbb{R}^{d_Y\times k}$, $b_2\in\mathbb{R}^{d_Y}$, such that
\begin{align*}
    &\|\Psi(\tb{u}_0) - \mathcal{R}^{d_Y}\circ (W_2\cdot\sigma(W_1\cdot\mathcal{E}^D(\tb{u}_0)^T)+ b_2)\|_Y <
   % &\leq \|\Psi(u_0)-\Psi(u^m_0)\|_2 + \|\Psi(u^m_0) - (W_2\cdot\sigma(W_1\cdot((u_0(x_0),u_0(x_1),\dots,u_0(x_m))^T)+ b_2)\|_2\\
    %&< C_1e^{CR^4\lambda_D}\kappa(D,m)+\epsilon-C_1e^{CR^4\lambda_D}\kappa(D,m) = 
    \epsilon.
\end{align*}
\end{theorem}
 %The image of $\Omega_X$ under $\Psi$ is also bounded. Let $$v\in \Omega_Y:=\{v\in Y~|~v=\Psi(u)~for ~u\in \Omega_X\},$$ we have $\|v\|_Y\leq L_\Psi R_X$. 
%Denote $\Psi_\gamma$ the push-forward measure of $\gamma$ under $\Psi$, such that for any $V\subset Y$, $$\Psi_\gamma(V) = \gamma(\{u:\Psi(u)\in \Omega\}).$$

\section{Learning operator to the weak solutions}\label{sec:main}

\subsection{Galerkin approximation.}\label{sec:Galerkin}

From now on, $ s=d/2$.
The global existence of Leray-Hopf solutions was established based on Galerkin approximation. 
The eigenfunctions $\{\tb{w}_k\}_{k\in\mathbb{Z}^d}$ of the Stokes operator $-\Delta$ form an orthonormal basis of $H$. Precisely, 
for each $k\in\mathbb{Z}^d\backslash \{0\}$, we choose $\beta_k,~\beta_{-k}\in \mathbb{R}^d$ such that 
$\{\beta_k\cos(k\cdot\tb{x})\}\bigcup\{\beta_k\sin(k\cdot\tb{x})\}$ forms the orthonormal basis and the corresponding eigenvalue is given by $\lambda_k = |k|^2$.
We re-label the sequence $\{\tb{w}_k\}_{k\in\mathbb{N}}$ such that the corresponding $\{\lambda_k\}_{k\in\mathbb{N}}$ is non-decreasing. 
%In particular, when $d=3$, Van der Corput’s estimate indicates that for $\tb{w}_m$, the relation of $m$ and $\lambda_m$ satisfies
%\begin{align}\label{van-der-corput}
%    m = \frac{4\pi}{3} \lambda_m^\frac{3}{2} + O(\lambda_m^\frac{3}{4}).
%\end{align}

For any $\tb{h}\in H$, $\tb{h}$ can be written as $\tb{h} = \sum^\infty_{i=1} \langle \tb{h}, \tb{w}_i\rangle \tb{w}_i.$
Given $m\in\mathbb{N}$, we 
define a projection operator $P_m : H\to \text{span}\{\tb{w}_1,\dots,\tb{w}_m\}$ by 
\begin{align*}
    P_m(\tb{h}) = \sum_{i=1}^m (\tb{h},\tb{w}_i)\tb{w}_i.
\end{align*}

The $m$-th order Galerkin approximation of Navier-Stokes equations is given by:
\begin{align}\label{Galerkin_equation}
    &\partial_t \tb{u}^m + P_m[(\tb{u}^m\cdot \nabla)\tb{u}^m] - \nu \Delta \tb{u}^m = 0,\nonumber\\
    & \tb{u}^m_0=P_m(\tb{u}_0).
\end{align}

Suppose the solution of (\ref{Galerkin_equation}) takes the form
\begin{align}\label{ICNSdef:u_m}
    \tb{u}^m(\tb{x},t) = \sum_{k=1}^m d^m_k(t)\tb{w}_k(\tb{x}) \in~ \text{span}\{\tb{w}_1,\dots,\tb{w}_m\}
\end{align}
%and $u^N(x,t)$ satisfies for almost every $t\in [0,T)$, 
%\begin{align} \label{u_m}
%    \langle \partial_t u^N,w_k \rangle + \nu a(u^N,w_k) + b(u^N,u^N, w_k) = 0~~~~~~~~k =1,2,\dots,N,    
%\end{align}
%with initial condition 
%\begin{align} \label{u_m:ini}
%    u^N(x,0) = \sum^N_{k=1}d^N_k(0)w_k(x), ~~\text{and}~~d^N_k(0) = \langle u_0(x),w_k(x)\rangle.
%\end{align}
with $m$ unknown functions $d^m_k, k=1,2,\dots,m$. We multiply (\ref{Galerkin_equation}) with $\tb{w}_k$ and integrate over $\tb{x}$.
By orthogonality of $\{\tb{w}_k\}$ in $L^2(\mathbb{T}^d)$ and the properties of eigenfunctions, we obtain an ODE system in $d^m_k(t),~k=1,2,\dots,m$:
\begin{align}\label{d^m_k}
    \dfrac{d}{dt} d^m_k(t) + \nu\lambda_k d^m_k(t) + \sum^m_{j,l=1}d^n_j(t)d^m_l(t)\int_{\mathbb{T}^d} (\tb{w}_j(\tb{x})\cdot\nabla)\tb{w}_l(\tb{x})\cdot \tb{w}_k(\tb{x})\text{d}\tb{x} = 0.
\end{align}

By classical ODE theory, there exists $T>0$ such that the ODE system (\ref{d^m_k}) admits unique solutions $\{d^m_k(t)\}^m_{k=1}\in C^1((0,T)).$
Let $t\in (0,T]$, multiply \ref{Galerkin_equation} by $\tb{u}^m$ and integrate over $\tb{x}$ to get
\begin{align*}
    \int_{\mathbb{T}^d} \partial_t \tb{u}^m\cdot \tb{u}^m~\text{d}\tb{x} + \int_{\mathbb{T}^d} (\tb{u}^m\cdot \nabla) \tb{u}^m\cdot \tb{u}^m~\text{d}\tb{x} - \nu\int_{\mathbb{T}^d} \Delta \tb{u}^m \cdot \tb{u}^m ~\text{d}\tb{x} = 0.
\end{align*}

By the incompressibility of $\tb{u}^m$, we have 
\begin{align}\label{b(u,u,u)}
    \int_{\mathbb{T}^d} (\tb{u}^m\cdot \nabla) \tb{u}^m\cdot \tb{u}^m~\text{d}\tb{x} = 0.
\end{align}

Thus we can obtain the following $m$-dimensional energy equality
\begin{align*}
    \dfrac{1}{2}\dfrac{d}{dt}\|\tb{u}^m(t)\|^2_{L^2(\mathbb{T}^d)} + \nu \|\nabla \tb{u}^m(t)\|^2_{L^2(\mathbb{T}^d)} = 0.
\end{align*}

%By Lax-Milgram theorem, $\exists~ C,c>0$ independent of $N$, such that $c\|u^N\|^2_V\leq a(u^N,u^N)\leq C\|u^N\|^2_V$. Using Cauchy-Schwarz inequality, Young's inequality, we have
%\begin{align}\label{u^N_bdd}
%    \dfrac{1}{2}\dfrac{d}{dt}\|u^N(t)\|^2_{(L^2(\Omega))^d} + C\epsilon\|u^N(t)\|^2_V \leq \dfrac{C\epsilon}{2}\|u^N(t)\|^2_V.
%\end{align}
Integrating on $t\in [0,T],~T>0$, this energy equality implies that 
\begin{align}\label{energy:u^m}
\begin{split}
   & \sup_{t\in [0,T]} \dfrac{1}{2}\|\tb{u}^m(t)\|^2_{L^2(\mathbb{T}^d)} \leq \dfrac{1}{2}\|\tb{u}_0\|^2_{L^2(\mathbb{T}^d)},
   \\&\nu\int^T_0 \|\nabla \tb{u}^m(t)\|^2_{L^2(\mathbb{T}^d)}\text{d}t \leq \dfrac{1}{2}\|\tb{u}_0\|^2_{L^2(\mathbb{T}^d)}.
   \end{split}
\end{align}
Therefore
$\{\tb{u}^m\}$ is uniformly bounded in $L^2(0,T;V) \cap L^\infty(0,T;H)$ for any positive $T$. In particular, $\{d^m_k(t)\}^m_{k=1}$ is uniformly bounded in $t$, hence $T$ can be $\infty$. 

By classical theory of weak solutions, $\{\tb{u}^m\}^\infty_{m=1}$ has a subsequence, still denoted by $\{\tb{u}^m\}^\infty_{m=1}$, that strongly converges to $\tb{u}\in L^2(0,T;V)$, and $\tb{u}$ is a Leray-Hopf weak solution to the incompressible Navier-Stokes equations \ref{ICNS}-\ref{ini_ICNS}. Readers may refer to \cite{PLLions} for more details.

\subsection{Properties of Approximating Operators}\label{sec:lips}

In this Section, $m\in\mathbb{N}$ is a fixed order in the Galerkin approximation.
We have seen in Section \ref{sec:Galerkin} 
that the first $m$ principal eigenfunctions $\{\tb{w}_1,\tb{w}_2,\dots,\tb{w}_m\}$ together with the projection operator $P_m$ provide a way to reduce the dimension of the target operator between infinite dimensional Hilbert spaces to a mapping between finite dimensional vector space.  The goal of this subsection is to show that the solution operator is Lipschitz in term of initial data.

Let $\tb{u}_0 \in H$, denote $\tb{u}^m_0$ to be
\begin{align}\label{u^m_0}
    \tb{u}^m_0 := P_m(\tb{u}_0) := \sum^m_{i=1}\langle \tb{u}_0, \tb{w}_i\rangle \tb{w}_i.
\end{align}
By Assumption \ref{ass:energy}, 
 $ \|\tb{u}^m_0\|_H \leq \|\tb{u}_0\|_H\leq R$.

Denote by $\Psi^m:\text{span}\{\tb{w}_1,\dots,\tb{w}_m\}\to Y$ the operator mapping $\tb{u}^m_0$ to $\tb{u}^m$, where $\tb{u}^m$ is the unique solution of $m$-th order Galerkin equation (\ref{Galerkin_equation}). 
 By energy inequality (\ref{energy:u^m}) and Assumption \ref{ass:energy},
$$\|\Psi^m(\tb{u}^m_0)\|_Y=\|\tb{u}^m\|_Y\leq C(\nu)\|\tb{u}_0\|_H\leq C(\nu)R,$$
where $C(\nu)$ is a positive constant depending on $\nu$, but independent of $m$.  The next lemma shows that $\Psi^m$ is a Lipschitz operator with respect to the initial function.
 
\begin{lemma}\label{ICNS:lips}
Let $d=2,3$, and $T>0$.
For any $m\in\mathbb{N}$, $$\Psi^m:\text{span}\{\tb{w}_1,\dots,\tb{w}_m\}\to Y$$ is a Lipschitz operator, i.e. for any $\tb{u}_0, \tb{v}_0\in H$ satisfying (\ref{icns:initial_energy}) and $\tb{u}^m_0,\tb{v}^m_0$ given in (\ref{u^m_0}), 
\begin{align}\label{ineq:lips}
    \|\Psi^m(\tb{u}^m_0) - \Psi^m(\tb{v}^m_0)\|_Y \leq L_{\Psi^m}\|\tb{u}^m_0-\tb{v}^m_0\|_H
\end{align}
for some $L_{\Psi^m}>0$.
\end{lemma}
\begin{proof}
Without any ambiguities, we omit the subscript of norms.
Let $\tb{u}_0,\tb{v}_0\in H$ with finite initial energy (\ref{icns:initial_energy}), we denote by $$\tb{u}^m:=\Psi^m(\tb{u}^m_0),\;\;\tb{v}^m:=\Psi^m(\tb{v}^m_0)$$the corresponding solutions respectively to the $m$-th order Galerkin approximation  (\ref{Galerkin_equation}) in the form \ref{ICNSdef:u_m}. 

Accordingly, we denote by $$\tb{h}^m_0=\tb{u}^m_0-\tb{v}^m_0,\;\;\text{ and }\;\;\tb{h}^m=\tb{u}^m-\tb{v}^m.$$
For any $s\in (0,T]$, multiply (\ref{Galerkin_equation}) at time $s$ with $\tb{h}^m(s)$ and integrate in $x$ to get 
\begin{align}\label{eq:uh}
    -\langle \tb{u}^m,\partial_t \tb{h}^m\rangle + \nu \langle \nabla \tb{u}^m,\nabla \tb{h}^m\rangle + \langle (\tb{u}^m\cdot \nabla)\tb{u}^m,\tb{h}^m\rangle =
    -\dfrac{d}{dt}\langle \tb{u}^m,\tb{h}^m\rangle,
\end{align}
Similarly, 
\begin{align}\label{eq:vh}
    -\langle \tb{v}^m,\partial_t \tb{h}^m\rangle + \nu \langle \nabla \tb{v}^m,\nabla \tb{h}^m\rangle + \langle (\tb{v}^m\cdot \nabla)\tb{v}^m,\tb{h}^m\rangle =
    -\dfrac{d}{dt}\langle \tb{v}^m,\tb{h}^m\rangle.
\end{align}
The difference of \ref{eq:uh} and \ref{eq:vh} gives
\begin{align}\label{eq:h}
    \dfrac{1}{2}\dfrac{d}{dt}\|\tb{h}^m\|^2 + \nu\|\nabla \tb{h}^m\|^2 &\leq \big| \langle (\tb{u}^m\cdot\nabla)\tb{u}^m-(\tb{v}^m\cdot\nabla)\tb{v}^m,\tb{h}^m\rangle\big| \nonumber\\
    &= \big | \langle (\tb{u}^m\cdot\nabla)\tb{u}^m,\tb{h}^m\rangle - \langle ((\tb{u}^m-\tb{h}^m)\cdot\nabla)(\tb{u}^m-\tb{h}^m),\tb{h}^m\rangle\big| \nonumber\\
    &= \big| \langle (\tb{u}^m\cdot\nabla)\tb{h}^m,\tb{h}^m\rangle + \langle (\tb{h}^m\cdot\nabla)\tb{u}^m,\tb{h}^m\rangle - \langle (\tb{h}^m\cdot \nabla)\tb{h}^m,\tb{h}^m\rangle\big |.
\end{align}

By (\ref{b(u,u,u)}) and the incompressibility of $\tb{h}^m$, the first and the last terms on the right side of \ref{eq:h} vanish. Thus we have the following inequality
\begin{align}\label{eq:h_1}
    \dfrac{1}{2}\dfrac{d}{dt}\|\tb{h}^m\|^2 + \nu\|\nabla \tb{h}^m\|^2 \leq \big|\langle (\tb{h}^m\cdot\nabla)\tb{u}^m,\tb{h}^m\rangle\big| 
    \leq \int_{\T^d} |\tb{h}^m|^2 |\nabla \tb{u}^m| d\tb{x}.
\end{align}

The Gagliardo-Nirenberg interpolation inequality with $\theta=\frac{d}{4}$ implies that
\begin{align*}
    \|\tb{h}^m\|^2_{L^4}\leq C\|\nabla \tb{h}^m\|^{d/2}\|\tb{h}^m\|^{2-d/2}
\end{align*}
holds for $d=2,3$. Then \ref{eq:h_1} becomes 
\begin{align*}
     \dfrac{1}{2}\dfrac{d}{dt}\|\tb{h}^m\|^2 + \nu \|\nabla \tb{h}^m\|^2 &\leq C\|\nabla \tb{u}^m\|\cdot \|\tb{h}^m\|^{2-d/2} \cdot \|\nabla \tb{h}^m\|^{d/2}\\
     &\leq C\|\tb{h}^m\|^2\cdot \|\nabla \tb{u}^m\|^{4/4-d} + \nu\|\nabla \tb{h}^m\|^2.
\end{align*}
Here we used Young's inequality and the constant $C$ depends on $d$ and $\nu$. 

%Since the space $\text{span}\{w_1,...,w_m\}$ is a finite dimensional space, we know that $L^2$ norm and $L^\infty$ norm are equivalent. There exists $C,c>0$ such that $c\|\nabla u\|\leq \|\nabla u\|_\infty\leq C\|\nabla u\|$. Thus \ref{eq:h_1} becomes
%\begin{align*}
%    \dfrac{1}{2}\dfrac{d}{dt}\|h\|^2 + \|\nabla h\|^2 \leq C\|\nabla u\|_\infty\|h\|^2
%\end{align*}
%for some generic constant $C$. Thus
%\[
%\dfrac{d}{dt}\|h\|^2 \leq C\|\nabla u\|_\infty\|h\|^2.
%\]
By Gronwall's inequality, we have
\begin{equation}\label{ICNS:Gronwall}
    \|\tb{h}^m\|^2 \leq \|\tb{h}^m_0\|^2 \exp\Big(C\int_0^s \|\nabla \tb{u}^m\|^{4/4-d} dt\Big). 
    %\leq \begin{cases}
    %\|h_0\|^2 \exp(C\|u_0\|^2_H)\leq \|h_0\|^2\exp(CR^2) &\text{for}~ d=2\\
    %\|h_0\|^2 \exp(C\|u_0\|^{\frac{4}{3}}_H\cdot s^{\frac{1}{3}}) \leq \|h_0\|^2\exp(CR^{\frac{4}{3}}T^{\frac{1}{3}})) &\text{for}~d=3.
    %\end{cases}
\end{equation}

When $d=2$,  one may use (\ref{energy:u^m}) and (\ref{icns:initial_energy}) such that (\ref{ICNS:Gronwall}) can be bounded by $\|\tb{h}^m_0\|^2\exp(CR^2)$ with $C$ depending on $\nu$ and $d$ but independent of $m$. In this case, 
by the definition of $\tb{h}^m$ and $\tb{h}^m_0$, $\Psi^m$ is a Lipschitz operator with Lipschitz constant $L_{\Psi^m}=\exp(CR^2)$ independent of $m$.

When $d=3$, (\ref{ICNS:Gronwall}) can be written as 
\begin{align}\label{I+II}
    \|\tb{h}^m\|^2 &\leq \|\tb{h}^m_0\|^2 \exp\Big(C\int_0^s \|\nabla \tb{u}^m\|^4 dt\Big)\nonumber\\
    &\leq \|\tb{h}^m_0\|^2 \exp\Big(C\underbrace{\sup_t \|\nabla \tb{u}^m(t)\|^2}_{\text{I}}\cdot \underbrace{\int^s_0 \|\nabla \tb{u}^m(t)\|^2 \text{d}t}_{\text{II}}\Big).
\end{align}

It is clear that II is uniformly bounded by $\|\tb{u}^m_0\|^2\leq R^2$ according to (\ref{energy:u^m}).  As for I, we use energy inequality (\ref{energy:u^m}) and the fact that $\{\tb{w}_1,\dots,\tb{w}_m\}$ is an orthonormal basis and the corresponding eigenvalue $\lambda_k$ is non-decreasing to get
\begin{align*}
    \sup_t \|\nabla \tb{u}^m(t)\|^2 &= \sup_t \|\sum^m_{k=1}d^m_k(t)\nabla \tb{w}_k\|^2\\
    &=\sup_t \sum^m_{k=1}\lambda_k (d^m_k(t))^2\\
    &\leq \lambda_m \sup_t\|\tb{u}^m(t)\|^2
    \leq \lambda_m \|\tb{u}^m_0\|^2 \leq \lambda_m R^2.
\end{align*}

Now (\ref{I+II}) becomes $$\|\tb{h}^m\|^2\leq \|\tb{h}^m_0\|^2\exp(CR^4\lambda_m)$$ with $C$ depending on $\nu$, $d$ and $\lambda_m$ depending on $m$. Therefore in 3D case, $\Psi^m$ is a Lipschitz operator with Lipschitz constant $$L_{\Psi^m} = \exp(CR^4\lambda_m)$$ depending on $m$.
\end{proof}
\begin{remark}\label{2D_psi_lips}
In 2-dimensional case, the Lipschitz constant $L_{\Psi^m}$ is independent of $m$, this implies that by taking $m\to\infty$, $\Psi$ is also Lipschitz with the same Lipschitz constant as each $\Psi^m$ . The Lipschitz condition on $\Psi$ provides an alternative approach to prove the uniqueness of 2D Leray-Hopf weak solution.
\end{remark}

\begin{remark}\label{3D_psi_lips}
    In 3-dimensional case, the Lipschitz constant $L_{\Psi^m}=\exp(CR^4\lambda_m)$ depends on $m$ as we arrange the basis functions $\{\tb{w}_k\}^\infty_{k=1}$ in an order that $\lambda_k$ is non-decreasing. The Galerkin approximating sequence $\{\tb{u}^m\}$ converges to $\tb{u}$ in $L^2(0,T;V)$ as $m\to\infty$, the corresponding $L_{\Psi^m}\to\infty$ exponentially fast. 
\end{remark}

\subsection{The Empirical Projection Error}\label{sec:ep}

The empirical projection errors on both $H$ and $Y$ spaces w.r.t. a finite collection of training data $S(N)=\{\tb{x}_i, \tb{y}_i\}^N_{i=1}$ can be bounded using a 
Monte Carlo type estimate. Given encoding dimension $d_H, d_Y>0$, the projection errors are denoted by:
\begin{align}\label{R(V)}
    \text{Pr}:=\mathbb{E}_{\tb{x}\sim\mathbb{P}}\|\tb{x} - \Pi_{d_H}\tb{x}\|^2_H,~~~~~~
    \text{Pr}^{\Psi_\sharp\mathbb{P}} := \mathbb{E}_{\tb{y}\sim\Psi_\sharp\mathbb{P}}\|\tb{y} - \Pi_{d_Y}\tb{y}\|^2_Y.
\end{align} 
And the corresponding empirical projection errors are given by:
\begin{align}\label{RN(V)}
    \text{Pr}(N) := \frac{1}{N}\sum^N_{j=1}\|\tb{x}_j - \Pi_{d_H}\tb{x}_j\|^2_H,~~~~~~
    \text{Pr}^{\Psi_\sharp\mathbb{P}}(N):=\frac{1}{N}\sum^N_{j=1}\|\tb{y}_j - \Pi_{d_Y}\tb{y}_j\|^2_Y.
\end{align}
We state the following lemma which is a partial result of Theorem 3.4 \cite{bhattacharya}, and the proof can be found therein. 
\begin{lemma}\label{lemma:empirical proj} There exist constants $Q_H, Q_Y\geq 0$ such that
\begin{align}\label{ineq:lemma ep}
    \mathbb{E}_{\tb{x}_j\sim\mathbb{P}} [\text{Pr}(N)] \leq \sqrt{\frac{Q_Hd_H}{N}} + \text{Pr},~ \text{and}~~
    \mathbb{E}_{\tb{y}_j\sim\Psi_\sharp\mathbb{P}} [\text{Pr}^{\Psi_\sharp\mathbb{P}}(N)] \leq \sqrt{\frac{Q_Yd_Y}{N}} + \text{Pr}^{\Psi_\sharp\mathbb{P}}.
\end{align}  
\end{lemma}

\subsection{ The Networks Approximation Error}
In this subsection, we prove an auxiliary lemma on the networks approximation error, together with projection errors in both $H,Y$ spaces we can prove our main results Theorem \ref{thm:ICNS_2D} and Theorem \ref{thm:ICNS_3D_lowdim}. This lemma is inspired by \cite{bhattacharya} Theorem 3.5, the difference comes from the choice of encoder and decoder. In \cite{bhattacharya}, the authors use PCA as an empirical dimension reduction method; while in our work, we use deterministic basis which is independent of training data.

Given $d_H,d_Y>0$, we define $\psi:\mathbb{R}^{d_H}\to \mathbb{R}^{d_Y}$ as
    \begin{align}\label{def:psi}
        \psi := \mathcal{P}^{d_Y}\circ \Psi\circ \mathcal{D}^{d_H},
    \end{align}
We aim to approximate $\psi$ by some ReLU neural networks $\Gamma:\mathbb{R}^{d_H}\to \mathbb{R}^{d_Y}$. We also define two operators
 $\Psi_\Delta, \Psi_{NN}:H\to Y$ as
\begin{align*}
    \Psi_\Delta:=\mathcal{R}^{d_Y}\circ \psi \circ \mathcal{E}^{d_H}, \\
    \Psi_{NN}:= \mathcal{R}^{d_Y}\circ \Gamma \circ \mathcal{E}^{d_H}.
\end{align*}

\begin{lemma}\label{lemma:auxi}
    Let $H, Y$ be two separable Hilbert spaces defined in Section \ref{sec:notation}.
    Let $\mathbb{P}$ be a probability measure on $B_R(H):=\{x\in H;~\|x\|_H\leq R\}$ with compact support.
    %and a finite collection $\{\tb{x}_i\}^N_{i=1}$ $N$ i.i.d. draws from $\mathbb{P}$ satisfying $\|\tb{x}_i\|_H\leq R$, for some $R>0$. 
    Let $\Psi:H\to Y$ be a Lipschitz operator with Lipschitz constant $L_\Psi$. 
    Fix $d_H,d_Y >0$, the encoders and decoders on $H,Y$: $\mathcal{E}^{d_H},\mathcal{D}^{d_H},\mathcal{P}^{d_Y}, \mathcal{R}^{d_Y}$ are defined in (\ref{encoder}),(\ref{decoder}) and (\ref{Y:encoder}). Let $\delta>0$, there exist a constant $c = c(d_H,d_Y)>0$ and a ReLU neural network $$\Gamma_{NN}\in\mathcal{F}_{NN}(d_H, d_Y, L, r, \frac{R}{\sqrt{d_H}})$$ such that 
    \begin{align*}
        \mathbb{E}_{\tb{x}_i\sim\mathbb{P}}\mathbb{E}_{\tb{x}\sim\mathbb{P}}\|\mathcal{R}^{d_Y}\circ \Gamma_{NN} \circ \mathcal{E}^{d_H}(\tb{x})-\Psi(\tb{x})\|^2_Y < 2\delta^2 + 4L_\Psi\bigg( \sqrt{\frac{Q_Hd_H}{N}} + \text{Pr}\bigg) + 4\bigg( \sqrt{\frac{Q_Yd_Y}{N}} + \text{Pr}^{\Psi_\sharp\mathbb{P}}\bigg)
    \end{align*}
    for some $Q_H, Q_Y\geq 0$, where
    \begin{align*}
        L\leq c \bigg[ \log{\bigg(\dfrac{R\sqrt{d_Y}}{\delta\sqrt{d_H}}\bigg) } + 1 \bigg]~,~~\text{and}~~~r\leq c\bigg(\dfrac{\delta\sqrt{d_H}}{2R}\bigg)^{-d_H} \bigg[ \log{\bigg(\dfrac{R\sqrt{d_Y}}{\delta\sqrt{d_H}}\bigg) } + 1 \bigg]~.
    \end{align*}

\end{lemma}

\begin{proof}
    
We simply write $\mathbb{E}_{\tb{x}_i\sim\mathbb{P}}\mathbb{E}_{\tb{x}\sim\mathbb{P}}$ as $\mathbb{E}\mathbb{E}$, and decompose $\mathbb{E}\mathbb{E}\|\Psi_{NN}(\tb{x})-\Psi(\tb{x})\|^2_Y$ as:
\begin{align}\label{ineq:I+II}                  \mathbb{E}\mathbb{E}\|\Psi_{NN}(\tb{x})-\Psi(\tb{x})\|^2_Y
    \leq 2\underbrace{\mathbb{E}\|\Psi_{NN}(\tb{x})-\Psi_\Delta(\tb{x})\|^2_Y}_{\text{I}} + 2\underbrace{\mathbb{E}\mathbb{E}\|\Psi_\Delta(\tb{x})-\Psi(\tb{x})\|^2_Y}_{\text{II}}.
\end{align}
We first look at II, recall the definition of $\Pi_{d_H}, \Pi_{d_Y}$ in (\ref{Pi_H}) and (\ref{Pi_Y}):
\begin{align}\label{ineq:II}
    \mathbb{E}\mathbb{E}\|\Psi_\Delta(\tb{x})-\Psi(\tb{x})\|^2_Y
    & = \mathbb{E}\mathbb{E}\|\mathcal{R}^{d_Y}\circ \mathcal{P}^{d_Y} \circ \Psi \circ \mathcal{D}^{d_H}\circ \mathcal{E}^{d_H} (\tb{x}) - \Psi(\tb{x})\|^2_Y \nonumber\\
    &= \mathbb{E}\mathbb{E} \|\Pi_{d_Y}\Psi(\Pi_{d_H}(\tb{x})) - \Psi(\tb{x})\|^2_Y \nonumber\\
    &\leq 2\mathbb{E}\mathbb{E}\|\Pi_{d_Y}\Psi(\Pi_{d_H}(\tb{x})) - \Pi_{d_Y}\Psi(\tb{x})\|^2_Y + 2\mathbb{E}\mathbb{E}\|\Pi_{d_Y}\Psi(\tb{x}) - \Psi(\tb{x})\|^2_Y\nonumber\\
    &\leq 2L_\Psi \mathbb{E}\mathbb{E}\|\Pi_{d_H}(\tb{x})-\tb{x}\|^2_H + 2\mathbb{E}\mathbb{E}_{\tb{y}\sim\Psi_\sharp\mathbb{P}}\|\Pi_{d_Y}\tb{y}-\tb{y}\|^2_Y\nonumber\\
    &= 2L_\Psi \mathbb{E}_{\tb{x}_j\sim\mathbb{P}}[\text{Pr}(N)] + 2\mathbb{E}_{\tb{y}_i\sim\Psi_\sharp\mathbb{P}}[\text{Pr}^{\Psi_\sharp\mathbb{P}}(N)]\nonumber\\
    &\leq 2L_\Psi\bigg( \sqrt{\frac{Q_Hd_H}{N}} + \text{Pr}\bigg) + 2\bigg( \sqrt{\frac{Q_Yd_Y}{N}} + \text{Pr}^{\Psi_\sharp\mathbb{P}}\bigg),
\end{align}
where we use Lemma \ref{lemma:empirical proj} in the last inequality. We next control I by first approximating $\psi$ by some ReLU networks. From (\ref{def:psi}), $\psi$ is a Lipschitz function with Lipschitz constant $L_\Psi$. For any $\tb{a}\in\mathbb{R}^{d_H}$, 
\begin{align*}
    \psi(\tb{a}) = (\psi^{(1)}(\tb{a}), \cdots, \psi^{(d_Y)}(\tb{a})), 
\end{align*}
with $\psi^{(j)}\in C(\mathbb{R}^{d_H};\mathbb{R}), j=1,\dots,d_Y.$ In order to satisfy Assumption \ref{ass:energy}, we require that
 $\textbf{a} = (a_1,\dots,a_{d_H})\in \big[-\frac{R}{\sqrt{d_H}},\frac{R}{\sqrt{d_H}}\big]^{d_H}.$ Next we perform a change of variables, define $\Tilde{\psi}^{(j)}:[0,1]^{d_H}\to\mathbb{R}$ by
\begin{align*}
    \Tilde{\psi}^{(j)}(\tb{s}) = \dfrac{\sqrt{d_H}}{2R}\psi^{(j)}\bigg( \dfrac{2R}{\sqrt{d_H}}\tb{s}-\dfrac{R}{\sqrt{d_H}}\bigg),~~\tb{s}\in [0,1]^{d_H}.
\end{align*}
Equivalently, 
\begin{align*}
    \psi^{(j)}(\tb{s}) = \dfrac{2R}{\sqrt{d_H}}\Tilde{\psi}^{(j)}\bigg( \dfrac{\sqrt{d_H}\tb{s} + R}{2R} \bigg),
\end{align*}
and one can check that $\Tilde{\psi}^{(j)}, j=1,2,\dots, d_Y$ is Lipschitz and share the same Lipschitz constant as $\psi^{(j)}$. Given $\delta>0$,  we apply \cite{yarotsky_error_2017} Theorem 1 to $\Tilde{\psi}^{(j)}$ for each $j$ and obtain ReLU networks $$\Tilde{f}^{(1)},\dots,\Tilde{f}^{(d_Y)}:[0,1]^{d_H} \to\mathbb{R} $$ such that 
\begin{align*}
    |\Tilde{f}^{(j)}(\tb{s}) - \Tilde{\psi}^{(j)}(\tb{s})| < \dfrac{\delta\sqrt{d_H}}{2R\sqrt{d_Y}}, ~~\text{for all}~~\tb{s}\in [0,1]^{d_H}.
\end{align*}
Denote by $L^{(j)}$ the depth of $\Tilde{f}^{(j)}$, and $r^{(j)}$ the number of weights and bias of $\Tilde{f}^{(j)}$. $L^{(j)}, r^{(j)}$ satisfy:
\begin{align*}
    L^{(j)} \leq c^{(j)}\bigg [ \log{\bigg(\dfrac{R\sqrt{d_Y}}{\delta\sqrt{d_H}}\bigg) } + 1\bigg ]~,~~~~~~~
    r^{(j)} \leq c^{(j)}\bigg(\dfrac{\delta\sqrt{d_H}}{2R}\bigg)^{-d_H}\bigg [ \log{\bigg(\dfrac{R\sqrt{d_Y}}{\delta\sqrt{d_H}}\bigg) } + 1\bigg ]
\end{align*}
for some $c^{(j)}>0$ depending on $d_H$. Then we define $f^{(j)}:\mathbb{R}^{d_H}\to\mathbb{R}$ by
\begin{align}\label{def:f}
    f^{(j)}(\tb{a}):=\dfrac{2R}{\sqrt{d_H}}\Tilde{f}^{(j)}\bigg(\dfrac{\sqrt{d_H}\tb{a} + R}{2R}\bigg)
\end{align}for any $\tb{a}\in \big[-\frac{R}{\sqrt{d_H}},\frac{R}{\sqrt{d_H}}\big]^{d_H},$
and $|(f^{(1)}(\tb{a}), \dots, f^{(d_Y)}(\tb{a})) - \psi(\tb{a})|_2 < \delta.$

Now, define $\Gamma_{NN}(\tb{a}):=(f^{(1)}(\tb{a}), \dots, f^{(d_Y)}(\tb{a})),$ and set $\Gamma_{NN}\equiv 0$ outside $\big[-\frac{R}{\sqrt{d_H}},\frac{R}{\sqrt{d_H}}\big]^{d_H}$. Such
$\Gamma_{NN}$ satisfies
\begin{align*}
    \sup_{\tb{a}\in \big[-\frac{R}{\sqrt{d_H}},\frac{R}{\sqrt{d_H}}\big]^{d_H}}{|\Gamma_{NN}(\tb{a})-\psi(\tb{a})|_2} < \delta.
\end{align*}
Finally, we can bound I as
\begin{align}\label{bound I}
    &~~~~~\mathbb{E}_{\tb{x}\sim\mathbb{P}}\|\Psi_{NN}(\tb{x})-\Psi_\Delta(\tb{x})\|^2_Y\nonumber\\
    & = \int_{B_R(H)} \|\Psi_{NN}(\tb{x})-\Psi_\Delta(\tb{x})\|^2_Y~ \text{d}\mathbb{P}(\tb{x})\nonumber\\
    & = \int_{B_R(H)} \|\mathcal{R}^{d_Y}\circ \Gamma_{NN}\circ \mathcal{E}^{d_H}(\tb{x}) - \mathcal{R}^{d_Y}\circ \psi\circ \mathcal{E}^{d_H}(\tb{x})\|^2_Y~ \text{d}\mathbb{P}(\tb{x})\nonumber\\
    & \leq \int_{B_R(H)}|\Gamma_{NN}(\mathcal{E}^{d_H}(\tb{x})) - \psi(\mathcal{E}^{d_H}(\tb{x}))|^2_2 ~\text{d}\mathbb{P}(\tb{x}) <\delta^2.
\end{align}
Here we used $L_{\mathcal{R}^{d_Y}}=1$ and $\mathbb{P}(B_R(H))\leq 1$. 
Combining (\ref{bound I}) with (\ref{ineq:I+II}) and (\ref{ineq:II}), we finish the proof.

\end{proof}

\subsection{Approximation Theory of Leray-Hopf Solutions}\label{sec:proof}

In this subsection, we combine Lemma \ref{lemma:empirical proj} and Lemma \ref{lemma:auxi} to prove Theorem \ref{thm:ICNS_2D} and Theorem \ref{thm:ICNS_3D_lowdim} on the existence of a finite ReLU neural network that can approximate the Leray-Hopf solutions of the incompressible Navier-Stokes equations. 

\subsubsection{2D Leray-Hopf Solutions}

%The proof of Theorem \ref{thm:ICNS} The key step is that the Lipschitz constant we derived in Lemma \ref{ICNS:lips} is independent of the index $m$, which allows us to search for all $\Gamma_m$ such that $D_Y\circ \Gamma_m\circ E_X$ approximates $\Psi^m$ in the same network class $\mathcal{F}_{NN}$. 

Let $\epsilon>0$. In the case that $d=2$, thanks to Lemma \ref{ICNS:lips} and Remark \ref{2D_psi_lips}, the solution operator $\Psi:H\to Y$ is a Lipschitz operator and $L_\Psi = e^{CR^2}$ with $C=C(\nu)$ and $\nu$ is the viscosity coefficient. Suppose we have a training dataset $S(N)=\{\tb{x}_i, \tb{y}_i\}^N_{i=1}$, 
we first bound the projection error on $H$.
%One may choose sufficient amount of data $n$, dimension $D$ of encoded vector space of $H$ and dimension $N$ of encoded vector space of $Y$ to control both the network approximation error and the projection error by the accuracy $\epsilon$: 
\begin{align*}
\text{Pr} := \mathbb{E}_{\tb{x}\sim\mathbb{P}}\|\tb{x}-\Pi_{d_H}(\tb{x})\|^2_H &= \int_{B_R(H)}\|\sum^\infty_{i=d_H + 1}(\tb{x}, \tb{w}_i)\tb{w}_i\|^2_H ~\text{d}\mathbb{P}(\tb{x})\\
&= \sum^\infty_{i=d_H + 1}\int_{B_R(H)} |(\tb{x}, \tb{w}_i)|^2~\text{d}\mathbb{P}(\tb{x})
\end{align*}
Thus $\text{Pr}$ is non-negative and monotone decreasing as $d_H\to \infty$. We can choose $d_H$ and $N$ large enough such that $\text{Pr} + \sqrt{\frac{Q_Hd_H}{N}}< \frac{1}{16L_{\Psi}}\epsilon$. 
%By Lemma \ref{lemma:empirical proj},
%\begin{align}\label{bound:proj1}
%    \mathbb{E}_{\tb{x}_i\sim\mathbb{P}}[R_N(V)] < \frac{1}{16L_\Psi}\epsilon.
%\end{align}
Similarly, we choose $d_Y$ and further enlarge $N$ such that $\text{Pr}^{\Psi_\sharp\mathbb{P}} + \sqrt{\frac{Q_Yd_Y}{N}}<\frac{1}{16}\epsilon$. Note that the choice of $d_H,d_Y$ is independent of $N$.
%Again, by Lemma \ref{lemma:empirical proj}, 
%\begin{align}\label{bound:proj2}
%    \mathbb{E}_{\tb{y}_i\sim \Psi_\sharp\mathbb{P}}[R^{\Psi_\sharp\mathbb{P}}_N(V)]< \frac{1}{16}\epsilon. 
%\end{align}
Let $\delta = \dfrac{\sqrt{\epsilon}}{2}$ in Lemma \ref{lemma:auxi}, we complete the proof of Theorem \ref{thm:ICNS_2D}.
%by combining Lemma \ref{lemma:auxi}, (\ref{bound:proj1}) and (\ref{bound:proj2}).
%Next, we bound the projection error on $Y$. If $\tb{u}\in Y, \tb{u}\sim\Psi_\sharp\mathbb{P}$, $\tb{u}$ is a Leray-Hopf solution. Let $d_Y>0$ be the dimension of the vector space after encoding $Y$, 

%and $\tb{u}^{d_Y}$ be the $d_Y$-order Galerkin approximation of $\tb{u}$. We have 
%\begin{align}\label{bound:proj2ineq}
    %\mathbb{E}_{\tb{u}\sim\Psi_\sharp\mathbb{P}}\|\Pi_{d_Y}(\tb{u})-\tb{u}\|^2_Y &\leq \int_{\Psi(B_R(H))}\|\Pi_{d_Y}(\tb{u})-\Pi_{d_Y}(\tb{u}^{d_Y})\|^2_Y ~\text{d}\Psi_\sharp\mathbb{P}(\tb{u}) + \int_{\Psi(B_R(H))}\|\Pi_{d_Y}(\tb{u}^{d_Y}) - \tb{u}\|^2_Y ~\text{d}\Psi_\sharp\mathbb{P}(\tb{u})\\
%    &= 2\int_{\Psi(B_R(H))} \|\tb{u}-\tb{u}^{d_Y}\|^2_Y ~\text{d}\Psi_\sharp\mathbb{P}(\tb{u}) \leq 2\|\tb{u}-\tb{u}^{d_Y}\|^2_Y.
%\end{align}
%Here we have used $L_{\mathcal{P}^{d_Y}} = L_{\mathcal{R}^{d_Y}} = 1$, $\Pi_{d_Y}(\tb{u}^{d_Y}) = \tb{u}_{d_Y}$ and $\Psi_\sharp\mathbb{P}(\Psi(B_R(H)))\leq 1$. Now choose $d_Y$ large enough such that 
%\begin{align}\label{bound:proj2}
    %\mathbb{E}_{\tb{x}_i\sim\mathbb{P}}\mathbb{E}_{\tb{u}\sim\Psi_\sharp\mathbb{P}}\|\Pi_{d_Y}(\tb{u})-\tb{u}\|^2_Y < \dfrac{1}{16}\epsilon.
%\end{align}
\subsubsection{3D Leray-Hopf Solutions}\label{3D}

The major difference between 2D and 3D cases is that, we do not have the uniform boundedness of the Lipschitz constant $L_{\Psi^k}, k=1,2,\dots$ as discussed in Remark \ref{2D_psi_lips} and Remark \ref{3D_psi_lips}. When the Galerkin approximation order $k\to\infty$, the corresponding Lipschitz constant $L_{\Psi^k}\to\infty$ exponentially fast for each $\Psi^k$. This property can cause an extremely slow convergence of the approximation error, even failure in convergence. To alleviate the curse of dimensionality, we propose Assumption \ref{ass:lowdim} as a special case in 3D setting.

%Under Assumption \ref{ass:lowdim}, the dimension of the vector space after encoding $H$ is less or equal to $D$. Follow the construction of $u^k$ for each Galerkin order $k$, $u^k(t)$ lies in $\text{span}\{w_1,w_2,\dots\,w_D\}$. Thus the dimension of the vector space after decoding $D$. 
Let $\epsilon>0$ and
$\tb{u}_0\in \text{span}\{\tb{w}_1,\tb{w}_2,\dots\,\tb{w}_D\}$, the Galerkin approximations produce a sequence of solutions $\{\tb{u}^1,\tb{u}^2,\dots,\tb{u}^D,\tb{u}^D,\dots\}$ with all the rest terms being $\tb{u}^D$. Extracting a convergent subsequence, the 3D Leray-Hopf solution $\tb{u}$ under Assumption \ref{ass:lowdim} is equal to $\tb{u}^D$.
By Lemma \ref{ICNS:lips}, $\Psi=\Psi^D$ is a Lipschitz operator in terms of the initial function. 
%We can obtain the dimension $\Tilde{d}_Y>0$ such that (\ref{bound:proj2}) holds. And we define
%\begin{align}
%    d_H := D ,~~~\text{and}~~~d_Y:=\max\{\Tilde{d}_Y,D\}.
%\end{align}

Let $d_H= D$, then the projection error on $H$ vanishes. By Lemma \ref{lemma:empirical proj}, we can choose
$\Tilde{d_Y}$ and $N$ large enough such that 
\begin{align*}
\text{Pr}^{\Psi_\sharp\mathbb{P}} + \sqrt{\frac{Q_Y\Tilde{d_Y}}{N}} < \dfrac{1}{8}\epsilon.
\end{align*}
Let $\delta = \dfrac{\sqrt{\epsilon}}{2}$ in Lemma \ref{lemma:auxi} and choose $d_Y:=\max\{\Tilde{d}_Y,D\}$, we finish the proof of Theorem \ref{thm:ICNS_3D_lowdim}.
%\begin{align*}
%    \mathbb{E}_{\mathcal{S}(n)}\mathbb{E}_{u_0\sim\mathbb{P}}\|D^D\circ \Gamma_{NN}\circ E^D(u_0) - \Psi(u_0)\|^2_Y < \epsilon.
%\end{align*}
%The parameters of the network architecture $\mathcal{F}(N,L,p,M)$ are given by
%\begin{align*}
%    L=O(\Tilde{L}\log\Tilde{L}), ~p=O(\Tilde{p}\log\Tilde{p}),~ M = \sqrt{D}Re^{cR^2},~c=c(\nu)
%\end{align*}
%and $\Tilde{L},\Tilde{p}$ satisfies
%\begin{align*}
%    \Tilde{L}\Tilde{p} = \Big\lceil D^{-\frac{D}{4+2D}}n^{\frac{D}{4+2D}}\Big\rceil.
%\end{align*}
\begin{remark}
   When the initial function lies in a finite-dimensional subspace of $H$, the Lipschitzness of $\Psi$ in terms of initial function implies the uniqueness of 3D weak solution to the ICNS. 
\end{remark}

\section{Data Generation in 3D}\label{sec:data}
%We revisit the approximation theory to the 3D Leray-Hopf solutions under Assumption \ref{ass:lowdim} that the initial function $u_0$ lies in a finite dimensional subspace of $H$. We explain in details on data generation of the  $u_0$ in such a setting. Inspired by DeepONet \cite{lu_deeponet_2021}, the initial function can be viewed as the input function of unknown operators $\Psi$. In \cite{lu_deeponet_2021}, the authors estimate the number of sensors to achieve a good accuracy with input functions being Gaussian process for identifying the first order ODE in 1D. Later \cite{lanthaler_error_2022} analyzed the approximation error of DeepONet, and projected input functions from infinite dimensional space into an $m$-dimensional vector space, where $m$ denotes the number of sensors. In particular, $m\gg 1$.
%The high dimensionality of the input vector space affects the efficiency in realistic learning tasks. We adopt the concept of \textit{sensor} in a different setting: the input functions are encoded using eigenfunctions of Stokes operator into a $D$-dimensional vector space, independent of $m$. Given $\epsilon>0$, we aim to estimate the number of sensors to achieve approximation error $\epsilon$ based on the \textit{universal approximation theorem} \cite{chen_universal_1995}.
In this section, we prove Theorem \ref{thm:low bound}, in which we investigate how many sensors are needed to achieve accuracy $\epsilon$ for identifying weak solutions to the ICNS using neural networks with 1 hidden layer. 

Let $d_H, d_Y >0$. According to Section \ref{3D}, we assume that $d_H= D$. The initial function $\tb{u}_0$ satisfies Assumption \ref{ass:energy} and Assumption \ref{ass:lowdim}, i.e.
\begin{align}\label{data_u0}
    \tb{u}_0 = \sum^D_{i=1} a_i \tb{w}_i,~~\text{with each}~~a_i:=\langle \tb{u}_0, \tb{w}_i\rangle,~\text{and}~~
    \|\tb{u}_0\|_H = \big(\sum^D_{i=1} a^2_i\big)^{\frac{1}{2}} \leq R
\end{align}
where $\tb{w}_1,\dots,\tb{w}_D$ are the first $D$ components of the orthonormal basis $\{\tb{w}_i\}^\infty_{i=1}$ of $H$. In practice, we generate $\tb{u}_0$ by randomly sampling $\tb{a}$ from the compact subset $$\textbf{a}:=(a_1,\dots,a_D)\in\bigg[-\frac{R}{\sqrt{D}},\frac{R}{\sqrt{D}}\bigg]^{D}.$$
%to ensure that $\tb{u}_0$ satisfies Assumption \ref{ass:energy}: 
%\begin{align}\label{ineq:a}
%    \|\tb{u}_0\|_H = \|\sum^D_{i=1} a_i \tb{w}_i\|_H = \big(\sum^D_{i=1} a^2_i\big)^{\frac{1}{2}} \leq R.
%\end{align}

%Let $S(N)=\{u_j\}^N_{j=1}$ be a training set sampled from the probability measure $\mathbb{P}$. For a fixed positive integer $D$, let $E^D:H\to \mathbb{R}^D$ be the encoder defined in () such that 

As we consider the 3D problem in periodic domain, one can choose uniformly $m+1$ points $\tb{x}_0,\tb{x}_1,\dots,\tb{x}_m \in [-\pi,\pi]^3$ as location of sensors. For each $\tb{w}_i\in C^\infty_c(\T^3)$, $i=1,2,\dots, D$, we define a function $\tb{w}^m_i(\tb{x})$ to be the trilinear interpolation of $\tb{w}_i(\tb{x})$ based on $m+1$ sensors $\tb{x}_0,\tb{x}_1,\dots,\tb{x}_m$.
Denote the operator mapping $\tb{w}_i$ to $\tb{w}^m_i$ by $T_m$, then $T_m$ is a continuous operator. 
Moreover, by Taylor expansion up to the second order, there exists a constant $\kappa(D,m)\sim \frac{D}{m}$ such that
\begin{align}
    \max_{\tb{x}\in [-\pi,\pi]^3} |\tb{w}_i(\tb{x})-\tb{w}^m_i(\tb{x})|\leq \kappa(D,m),
\end{align}
and $\kappa(D,m)\to 0$ as $m\to\infty$.
For each $\tb{u}_0$ in (\ref{data_u0}), we define $\tb{u}^m_0$ based on $m$ sensors as:
\begin{align}\label{u^m_0:sensor}
    \tb{u}^m_0(\tb{x}) := \sum^D_{i=1}a_i \tb{w}^m_i(\tb{x}).
\end{align}

Let $\Psi:H\to Y$ be the operator discussed in Section \ref{3D}. Then $\Psi$ is a Lipschitz operator with $$L_\Psi = L_{\Psi^D} = e^{CR^4\lambda_D}$$ where $C = C(\nu)$ and $\nu$ is the viscosity coefficient. Let  $\mathcal{D}^D:\mathbb{R}^D\to H$ be the decoder on $H$ defined in (\ref{decoder}) and  $\mathcal{P}^{d_Y}: Y\to\mathbb{R}^{d_Y}$ be the encoder on $Y$ defined in (\ref{Y:encoder}). By Lemma \ref{encoder lips}, $\mathcal{P}^{d_Y}$ is Lipschitz with Lipschitz constant 1. We can derive that
\begin{align}\label{psi_error}
    &~~~~\|\mathcal{P}^{d_Y}\circ\Psi\circ\mathcal{D}^D(\tb{a})-\mathcal{P}^{d_y}\circ\Psi\circ T_m\circ\mathcal{D}^D(\tb{a})\|_2 \nonumber\\
    &\leq e^{CR^4\lambda_D} \|\tb{u}_0-\tb{u}^m_0\|_H 
    \leq C_1e^{CR^4\lambda_D}\kappa(D,m),
\end{align}
where $C_1=C_1(D,R)>0$.
Let $V\subset \text{span}\{\tb{w}_1,\dots,\tb{w}_D\}$ such that $V$ contains all the possible initial functions $\tb{u}_0$. Then $V$ is a compact subset as it is the image of a continuous map on a compact domain:  $$V = \bigg\{\mathcal{D}^D(\tb{a}): \tb{a}\in \bigg[-\frac{R}{\sqrt{D}},\frac{R}{\sqrt{D}}\bigg]^D \bigg\}.$$
We define 
\begin{align*}
    G_m &:= \{(\tb{u}_0(\tb{x}_0),\tb{u}_0(\tb{x}_1),\dots,\tb{u}_0(\tb{x}_m))^T~|~\tb{u}_0\in V\}\subset\mathcal{M}_\mathbb{R}(m+1, 3),\\
    V^m &:=\{\tb{u}^m_0 ~|~\tb{u}_0\in V\},
\end{align*}
where $\mathcal{M}_\mathbb{R}(m+1, 3)$ denotes the set of all real $(m+1)\times 3$ matrices. 
%$G_m$ is a compact subset of $\mathcal{M}_\mathbb{R}(m+1, 3)$ because each $\tb{u}_0$ is a linear combination of continuous maps $\tb{w}_i$ and the evaluation map  ...
Naturally there is a bijection between $G_m$ and $V^m$, and $\mathcal{D}^D$ is also a bijection between $\big[-\frac{R}{\sqrt{D}},\frac{R}{\sqrt{D}}\big]^D$ and $V$. 

Define $\phi_m:\mathbb{R}^D \to \mathbb{R}^{d_Y}$ 
by
\begin{align*}
    \phi_m(\tb{a}) := \mathcal{P}^{d_Y}\circ\Psi\circ T_m\circ\mathcal{D}^D(\tb{a}).
\end{align*}

Then $\phi_m$ is a continuous function. Given an approximation error $\epsilon$, choose the number of sensors $m$ such that the right hand side of (\ref{psi_error}) satisfies
\begin{align}\label{ineq:m}
    C_1e^{CR^4\lambda_D}\kappa(D,m)<\frac{1}{2} \epsilon.
\end{align}
We can apply the universal approximation theorem \cite{chen_universal_1995} which states that there exist a ReLU neural network with weights and bias $$W_1\in \mathbb{R}^{k\times D},\;\; b_1\in\mathbb{R}^D,\;\; W_2\in\mathbb{R}^{{d_Y}\times k},\;\;b_2\in\mathbb{R}^{d_Y},$$ such that
\begin{align*}
    \|\phi_m(\tb{a}) - (W_2\cdot\sigma(W_1\cdot\textbf{a}^T)+ b_2)\|_2
    <\dfrac{1}{2} \epsilon - C_1e^{CR^4\lambda_D}\kappa(D,m)
\end{align*}
holds for all $\tb{a}\in\big[-\frac{R}{\sqrt{D}},\frac{R}{\sqrt{D}}\big]^D$, then we have
\begin{align}\label{ineq:p,psi,d}
    &~~~~~\|\mathcal{P}^{d_Y}\circ\Psi\circ\mathcal{D}^D(\tb{a}) - (W_2\cdot\sigma(W_1\cdot\textbf{a}^T)+ b_2)\|_2\nonumber\\
    &\leq \|\mathcal{P}^{d_Y}\circ\Psi\circ\mathcal{D}^D(\tb{a})-\mathcal{P}^{d_Y}\circ\Psi\circ T_m\circ\mathcal{D}^D(\tb{a})\|_2\nonumber\\
    &~~~~~+ \|\mathcal{P}^{d_Y}\circ\Psi\circ T_m\circ\mathcal{D}^D(\tb{a}) - (W_2\cdot\sigma(W_1\cdot\textbf{a}^T)+ b_2)\|_2\nonumber\\
    &< C_1e^{CR^4\lambda_D}\kappa(D,m)+\dfrac{1}{2}\epsilon-C_1e^{CR^4\lambda_D}\kappa(D,m) = \dfrac{1}{2}\epsilon.
\end{align}
Suppose $d_Y$ is chosen sufficiently large such that for any $\tb{u}_0\in V$,
\begin{align}\label{ineq:proj sec4}
    \|\Psi(\tb{u}_0)-\Pi_{d_Y}(\Psi(\Pi_D(\tb{u}_0)))\|_Y&\leq \|\Psi(\tb{u}_0)-\Pi_{d_Y}(\Psi(\tb{u}_0)) \|_Y+ \|\Pi_{d_Y}(\Psi(\tb{u}_0)) - \Pi_{d_Y}(\Psi(\Pi_D(\tb{u}_0)))\|_Y\nonumber\\
    &= \|\Psi(\tb{u}_0)-\Pi_{d_Y}(\Psi(\tb{u}_0)) \|_Y <\dfrac{1}{2}\epsilon. 
\end{align}

Recall the definition of $\mathcal{E}^{d_H},\mathcal{R}^{d_Y}$ in (\ref{encoder})-(\ref{Y:encoder}) and Lemma \ref{encoder lips}, we combine (\ref{ineq:p,psi,d}) and (\ref{ineq:proj sec4}) to get
\begin{align*}
    &~~~~~\|\Psi(\tb{u}_0) - \mathcal{R}^{d_Y}\circ (W_2\cdot\sigma(W_1\cdot\mathcal{E}^D(\tb{u}_0)^T)+ b_2)\|_Y\\
    &\leq \|\Psi(\tb{u}_0)-\Pi_{d_Y}(\Psi(\Pi_D(\tb{u}_0)))\|_Y + \|\Pi_{d_Y}(\Psi(\Pi_D(\tb{u}_0))) - \mathcal{R}^{d_Y}\circ (W_2\cdot\sigma(W_1\cdot\mathcal{E}^D(\tb{u}_0)^T)+ b_2)\|_Y\\
    &< \dfrac{1}{2}\epsilon + \dfrac{1}{2}\epsilon = \epsilon
\end{align*}
holds for any $\tb{u}_0 = \mathcal{D}^D(\tb{a})\in V$, and $\tb{a}\in \big[-\frac{R}{\sqrt{D}},\frac{R}{\sqrt{D}}\big]^D$. Therefore, we conclude that the number of sensors $m$ should at least satisfy (\ref{ineq:m}) to achieve an approximation accuracy $\epsilon$.

%\bibliography{op_bib}
%\bibliographystyle{numeric}
\nocite{*}
\printbibliography
\end{document}